\documentclass[12pt, reqno]{amsart}
\usepackage{latexsym,amsmath,amssymb, cases, verbatim, esint}
\usepackage[active]{srcltx}
\usepackage[colorlinks]{hyperref}

\title[Existence for Dynamic Programming Principle]{\protect{General Existence of Solutions to Dynamic Programming Principle}}

\author[Q. Liu]{Qing Liu} 
\address{Qing Liu, Department of Mathematics, University of Pittsburgh, Pittsburgh, PA 15260, USA, {\tt qingliu@pitt.edu}}

\author[A. Schikorra]{Armin Schikorra}
\address{Armin Schikorra, Max-Planck Institut MiS Leipzig, Inselstr. 22, 04103 Leipzig, Germany, {\tt armin.schikorra@mis.mpg.de}}

\thanks{A.S. was supported by DAAD fellowship D/12/40670.}

\usepackage[margin=1.5in]{geometry}

plus 6pt minus 12pt
\def\ilap{\Delta_\infty}
\def\eps{\varepsilon}

\def\R{{\mathbb R}}

\newcommand{\tr}{\operatorname{tr}}
\newcommand{\dive}{\operatorname{div}}

\newtheorem{theorem}{Theorem}
\newtheorem{lemma}[theorem]{Lemma}
\newtheorem{corollary}[theorem]{Corollary}
\newtheorem{proposition}[theorem]{Proposition}

\theoremstyle{definition}
\newtheorem{remark}[theorem]{Remark}
\newtheorem{definition}[theorem]{Definition}
\newtheorem{example}[theorem]{Example}

\def\diam{{\rm diam\,}}
\def\dist{{\rm dist\,}}

\def\mvint_#1{\mathchoice
          {\mathop{\vrule width 6pt height 3 pt depth -2.5pt
                  \kern -8pt \intop}\nolimits_{\kern -3pt #1}}%
          {\mathop{\vrule width 5pt height 3 pt depth -2.6pt
                  \kern -6pt \intop}\nolimits_{#1}}%
          {\mathop{\vrule width 5pt height 3 pt depth -2.6pt
                  \kern -6pt \intop}\nolimits_{#1}}%
          {\mathop{\vrule width 5pt height 3 pt depth -2.6pt
                  \kern -6pt \intop}\nolimits_{#1}}}

\numberwithin{theorem}{section} \numberwithin{equation}{section}

\begin{document}

\sloppy


\sloppy


\begin{abstract}

We provide an alternative approach to the existence of solutions to dynamic programming equations arising in the discrete game-theoretic interpretations for various nonlinear partial differential equations including the infinity Laplacian, mean curvature flow and Hamilton-Jacobi type. Our general result is similar to Perron's method but adapted to the discrete situation. 

\end{abstract}


\maketitle

\section{Introduction}

Recently, the deterministic or stochastic game-theoretic approach to various nonlinear partial differential equations has attracted a lot of attention; see, for example, a variety of games in \cite{KS1, KS2, PSSW, PS, GL, ASS, PPS, APSS, CGR} and related topics in \cite{LA, MPR1, MPR2, RV}. The results in the literature so far can be summarized in the following way. Consider a partial differential equation in a bounded domain $\Omega\subset \mathbb{R}^n$:
\begin{equation}\label{E1}
\begin{cases}
F(x, \nabla u, \nabla^2 u)=0 \quad & \text{ in $\Omega$},\\
u=g \quad & \text{ on $\partial \Omega$,} 
\end{cases}
\end{equation}
where $F$ is a function defined on $\overline{\Omega}\times \mathbb{R}^n\times \mathcal{S}^n$ and $g$ is a continuous function on $\partial \Omega$. Here $\mathcal{S}^n$ denotes the set of all $n\times n$ symmetric matrices. Suppose there exists a unique (viscosity) solution $u$ of the equation. Then (under certain additional conditions) one may construct a deterministic or stochastic discrete game, with step size $\eps>0$, whose value function $u^\eps$ converges to $u$ (locally) uniformly as $\eps\to 0$.

The most important step in the proof of these game-theoretic approximations is to establish the so-called \emph{dynamic programming principle} \eqref{eq:TDPP} for each particular $F$. More precisely, one proves that $u^\eps$ satisfies a relation between its value at the current position and that of the next possible positions, which is reminiscent of nonlinear semigroup-properties. A formal derivation for (\ref{E1}) from the games often follows by adopting the Taylor expansion on its dynamic programming principle. 

In this work, we are not interested in sending $\eps\to 0$ and applications to the continuum equation, but instead in obtaining solutions to general discrete dynamic programming equations with an alternative PDE approach. 

One example we have in mind is related to the Dirichlet boundary problem for the infinity Laplacian with continuous boundary data $g: {\partial \Omega} \to \R$:
 \begin{numcases}{}
  -\ilap u  :=-\tr\left(\frac{\nabla u\otimes \nabla u}{|\nabla u|^2}\nabla^2 u\right)=  0 \quad &\mbox{in $\Omega$,}\nonumber \\ 
  u = g \quad & \mbox{on $\partial \Omega$}, \nonumber
 \end{numcases}
whose unique viscosity solution $u$ can be approximated locally uniformly by solutions $u^\eps$ to the DPP associated with a ``tug-of-war'' game, first proposed in \cite{PSSW},
\begin{equation}\label{eq:DPPinfty}
 \begin{cases}\displaystyle
  u^\eps(x) = \frac{1}{2} \sup\limits_{B_\eps(x)} u^\eps + \frac{1}{2} \inf\limits_{B_\eps(x)} u^\eps \quad &\mbox{for $x \in \Omega$}\\
  u^\eps = g \quad & \mbox{on $\partial \Omega$}.
 \end{cases}
\end{equation}

Let us define sub- and supersolutions for this particular problem.
\begin{definition}\label{def:subsubinftysols}
A function $u: \overline{\Omega} \to \R$ with $\sup_{\overline{\Omega}} u < \infty$ is called a subsolution of \eqref{eq:DPPinfty}, if it satisfies
\[
 \begin{cases}\displaystyle
  u(x) \leq \frac{1}{2} \sup\limits_{B_\eps(x)} u + \frac{1}{2} \inf\limits_{B_\eps(x)} u \quad &\mbox{for $x \in \Omega$},\\
  u \leq g \quad & \mbox{on $\partial \Omega$}.
 \end{cases}
\]
The set of all subsolutions is denoted by $\underline{\mathcal{S}}$.

Conversely, a function $u: \overline{\Omega} \to \R$ with $\inf_{\overline{\Omega}} u > -\infty$ is called a supersolution of \eqref{eq:DPPinfty}, if
\[
 \begin{cases}\displaystyle
  u(x) \geq \frac{1}{2} \sup\limits_{B_\eps(x)} u + \frac{1}{2} \inf\limits_{B_\eps(x)} u \quad &\mbox{for $x \in \Omega$},\\
  u \geq g \quad & \mbox{on $\partial \Omega$}.
 \end{cases}
\]
The set of all supersolutions is denoted by $\overline{\mathcal{S}}$.

A function $u: \overline{\Omega} \to \R$ is a solution, if it is simultaneously a subsolution and a supersolution. In the literature this is also called an infinity harmonious function, cf. \cite{LA}. The set of all solutions is denoted by $\mathcal{S}$.
\end{definition}

A first easy observation is that $\underline{\mathcal{S}}$ and $\overline{\mathcal{S}}$ are non-empty.
\begin{example}\label{ex:subspersolutions}
The function $u: \overline{\Omega} \to \R$ given by
\[
 u(x) := \begin{cases}
	  \inf_{\partial \Omega} g \quad &x \in \Omega\\
          g \quad &x \in \partial \Omega.
         \end{cases}
\]
is a subsolution, the function $v: \overline{\Omega} \to \R$ given by
\[
 v(x) := \begin{cases}
	  \sup_{\partial\Omega} g \quad &x \in \Omega\\
          g \quad &x \in \partial \Omega.
         \end{cases}
\]
is a supersolution.
\end{example}

The next observation is that any subsolution is uniformly bounded from above, and any supersolution is uniformly bounded from below. This follows from the following maximum principle:
\begin{proposition}[Maximum principle]
Any subsolution $u \in \underline{\mathcal{S}}$ satisfies 
\[
 \sup_{\overline{\Omega}} u \leq \sup_{\partial \Omega} g,
\]
and any supersolution 
$u \in \overline{\mathcal{S}}$ satisfies 
\[
 \inf_{\overline{\Omega}} u \geq \inf_{\partial \Omega} g.
\]
\end{proposition}
\begin{proof}
We are only going to show the subsolution case. It suffices to prove that for any $\nu > 0$,
\begin{equation}\label{eq:supremumbounded}
 \sup_{\overline{\Omega}} u \leq \sup_{\partial \Omega} g + \nu.
\end{equation}
So fix $\nu > 0$, and set $K := \lceil \frac{\diam \overline{\Omega}}{\eps} \rceil + 1$. For $\delta := (K+1)^{-1} \eps$, let $x_0 \in \overline{\Omega}$ 
\[
 \sup_{\overline{\Omega}} u \leq u(x_0) + \delta.
\]
If $x_0 \in \partial \Omega$, then \eqref{eq:supremumbounded} is proven. If $x_0 \in \Omega$, we pick a sequence $(x_i)_{i=0}^K$ such that $x_i \in B_\eps(x_{i-1})$ and $x_K \in \partial \Omega$. Since $u$ is a subsolution,
\begin{align*}
 u(x_0) &\leq \frac{1}{2} \sup_{B_\eps(x_0)} u + \frac{1}{2} \inf_{B_\eps(x_0)} u \leq \frac{1}{2} u(x_0) + \frac{1}{2} u(x_1) + \frac{1}{2}\delta.
\end{align*}
Consequently,
\[
 \sup_{\overline{\Omega}} u \leq u(x_0)  + \delta \leq u(x_1) + 2\delta.
\]
By induction,
\[
 \sup_{\overline{\Omega}} u \leq u(x_i) + (i+1) \delta,
\]
and in particular, since $x_K \in \partial \Omega$, by the choice of $\delta$,
\[
 \sup_{\overline{\Omega}} u \leq u(x_K) + (K+1) \delta \leq \sup_{\partial \Omega} g + \nu.
\]
\end{proof}

Having this boundedness, we can define the largest subsolution
\[
 \overline{u}(x) := \sup_{u \in \underline{\mathcal{S}}} u(x), \quad x \in \overline{\Omega},
\]
and the smallest supersolution
\[
 \underline{u}(x) := \inf_{u \in \overline{\mathcal{S}}} u(x), \quad x \in \overline{\Omega}.
\] 

Then we have the following result, which extends the classical Perron's method used in viscosity solution theory \cite{I1, CIL}.
\begin{theorem}[Perron's Method]\label{th:perron}
Both $\overline{u}$ and $\underline{u}$ are solutions to \eqref{eq:DPPinfty}.
\end{theorem}
\begin{proof}
We only give a proof for $\overline{u}$. It is easy to check that $\overline{u} \in \underline{\mathcal{S}}$. Indeed, for any subsolution $u$, we have $u\leq g$ on $\partial \Omega$ and 
\[
u(x)\leq {1\over 2}\sup_{B_\eps(x)} u+{1\over 2}\inf_{B_\eps(x)} u \text{ for any $x\in \Omega$.}
\]
A pointwise supremum yields $\overline{u}\leq g$ on $\partial \Omega$ and
\[
\overline{u}(x)\leq {1\over 2}\sup_{B_\eps(x)}\overline{u}+{1\over 2}\inf_{B_\eps(x)}\overline{u} \text{ for any $x\in \Omega$.}
\]

It suffices to prove that $\overline{u} \in \overline{\mathcal{S}}$. Since $\overline{u}$ is a subsolution, so is $v: \overline{\Omega} \to \R$ given by
\[
 v(x) := \begin{cases}
  \frac{1}{2} \sup\limits_{B_\eps(x)} \overline{u} + \frac{1}{2} \inf\limits_{B_\eps(x)} \overline{u} \quad &\mbox{for $x \in \Omega$}\\
  g \quad & \mbox{on $\partial \Omega$}.
 \end{cases}
\]
Since $\overline{u}$ is the pointwise supremum of all subsolutions, in particular
\[
 \frac{1}{2} \sup\limits_{B_\eps(x)} \overline{u} + \frac{1}{2} \inf\limits_{B_\eps(x)} \overline{u} = v(x) \leq \overline{u}(x),
\]
for $x\in \Omega$ and $g\leq \overline{u}$ on $\partial\Omega$, we deduce that $\overline{u}$ is a supersolution. Consequently, $\overline{u}$ is a solution.

\end{proof}
To our best knowledge, Perron's method (with envelope techniques) was first introduced for the existence of solutions to discrete equations by Armstrong and Smart \cite{AS2}, where they treated a slight variant of dynamic programming equation for the infinity Laplacian so that the solution enjoys better regularity. 

Our idea is similar to theirs but we study the DPP associated with the original tug-of-war posed by \cite{PSSW}, whose solutions are not expected to be even semicontinuous; see \cite[Example 2.4]{AS2} and also \cite[Example 1.2]{LS1}. The loss of regularity here is due to the jump of values on the boundary caused by the finite step size $\eps$. In this case, classical Perron's argument, often used in partial differential equation theory \cite{I1, CIL} or for modified dynamic programming equations \cite{AS2}, should be slightly adapted, as easily seen from the example above. One needs to, for example, only take the supremum of all subsolutions without further applying the semicontinuous envelopes. 

Our arguments leading to Theorem~\ref{th:perron} are very general. Indeed, they allow us to give an elementary approach for finding solutions to DPPs for a very general class of equations in a general metric space $X$ (equipped with a measure when necessary); more precisely, for $\mathcal{U} \subset \{u: X \to \R\}$ and a given operator $T: \mathcal{U} \to \mathcal{U}$ we find solutions $u$ to
\begin{equation}\label{eq:DPPwithT}
  u(x) = Tu(x) \quad \mbox{for $x \in X$}.
\end{equation}
In our case above, $T$ is given by
\begin{equation}\label{eq:DPPinftyOperator}
 \begin{cases}
  Tu(x)  = \frac{1}{2} \sup\limits_{B_\eps(x)} u + \frac{1}{2} \inf\limits_{B_\eps(x)} u \quad &\mbox{for $x \in \Omega$,}\\
  Tu(x) = g(x) \quad & \mbox{for $x \in \partial \Omega$}.
 \end{cases}
\end{equation}

The main properties for $T$ and $\mathcal{U}$ are as follows:
\begin{enumerate}
\item[(H1)] The operator $T$ is monotone.\\
\item[(H2)] There exists at least one subsolution (or supersolution).\\
\item[(H3)] There is a uniform upper bound for all subsolutions (or a uniform lower bound for all supersolutions).\\
%
\item[(H4)] The function space $\mathcal{U}$ is closed under the pointwise supremum operations (or infimum operations).\\
%
\end{enumerate}
We specify the precise meaning of these conditions in Section~\ref{generalcase}, and discuss examples of known DPPs and there respective operator $T$ in Section~\ref{section:example}.

Our main result, Theorem~\ref{main:thm2}, is that if $T$ satisfies (H1)--(H4), then there exists a solution to \eqref{eq:DPPwithT}.

Usually, the game-theoretic approach and Perron's method are seen as separate ways to obtain existence of solutions to fully nonlinear PDEs. Our result indicates that they are connected on the discrete level.

It is not unusual to assume the monotonicity (H1). Similar conditions are used in schemes for fully nonlinear second order PDE, see, e.g. \cite{BS}, which are elliptic, i.e.,
\begin{equation}\tag{E}
\begin{aligned}
F(x, p, A)\leq F(x, p, B) &\quad \mbox{for all $x\in \overline{\Omega}=\Omega\cup \partial \Omega$, $p\in \mathbb{R}^n$ and }\\
  &\mbox{$n\times n$ symmetric matrices $A, B$ with $A\geq B$}.
\end{aligned}
\end{equation}
%
The upper bound for subsolutions (H3) is necessary as to justify the pointwise supremum of subsolutions. It is, however, non-trivial to obtain the boundedness. In this work, we construct a bounded strict supersolutions and prove a comparison principle for any subsolution and strict supersolution, Theorem~\ref{comparison:thm}. We show the construction of strict supersolutions for each example in Section \ref{section:example}.

Concerning (H4), we must choose a suitable function space $\mathcal{U}$ to guarantee the closedness. The classical choice for elliptic and parabolic PDEs is the space of all upper/lower semicontinuous solutions and an application of semicontinuous envelope is also necessary; see for example \cite{I1, CIL}. As is already mentioned above, in our case, we generally cannot expect even semicontinuity. In the case of integral DPPs, $\mathcal{U}$ has to be a subset of all measurable functions, which makes it more difficult to ensure (H4).

The remaining issue is the uniqueness of solutions to DPP, which is equivalent to the usual comparison principle (as shown in Lemma~\ref{la:uniquenesVScomparison}), since we have Perron's method in hand for existence of solutions. Once the uniqueness is established,  we may conclude that the solutions we found via Perron's method coincide with the value function obtained in games. In \cite{LS1} we obtained uniqueness for a particular problem related to the biased tug-of-war game with nonzero running payoff. In Section \ref{sec uniqueness}, we generalize this to dynamic programming equations in the presence of positive or negative running payoffs. 

The uniqueness problem in general, especially without running payoffs, is still not known. If the semicontinuity of subsolutions is known, then a comparison principle of sub- and supersolutions can be proved; see \cite{AS, AS2}. We however do not have even semicontinuity. Another earlier related result for existence and uniqueness is due to \cite{LA} but for a quite different DPP; the radius $\eps$ of the ball $B_\eps(x)$ where the extrema are taken depends on $x$ and diminishes as $x$ approaches the boundary of the domain. In our general setting of \eqref{eq:DPPinfty}, one might hope that a finer analysis of game-trees as in \cite{LS1} leads to answers related to this question.

\subsection*{Acknowledgments}
The authors thank Juan J. Manfredi and Marta Lewicka for their interests and valuable suggestions. The authors also thank Scott Armstrong for his helpful remarks on the first draft of this paper.

\section{General Existence}\label{generalcase}
We consider a function space $\mathcal{U} \subset  \{u: X  \to \R\}$, and an operator $T: \mathcal{U} \to \mathcal{U}$. Then, we look for solutions $u \in \mathcal{U}$ to
\begin{equation}\label{eq:TDPP}
 Tu = u \quad \mbox{in $X$}.
\end{equation}

Let us start defining sub- and supersolutions to \eqref{eq:TDPP}. The following definition is consistent with the usual classical definition of sub- and supersolutions of PDEs with Dirichlet boundary conditions.
\begin{definition}
Given $T:\mathcal{U}\to \mathcal{U}$, a function $u\in \mathcal{U}$ is called a subsolution of \eqref{eq:TDPP} if $\sup_X u < \infty$ and
\[
u\leq Tu \quad \mbox{in $X$}.
\]
Similarly, a function $u\in \mathcal{U}$ is called a supersolution of \eqref{eq:TDPP} if $\inf_X u > -\infty$ and
\[
u\geq Tu \quad \text{in $X$}.
\]
A function $u\in \mathcal{U}$ is called a solution if it is both a subsolution and a supersolution, i.e., if $Tu = u$ pointwise in $X$.
\end{definition}

In the case of tug-of-war \eqref{eq:DPPinfty}, this definition is consistent with Definition~\ref{def:subsubinftysols}, with $T$ defined as in \eqref{eq:DPPinftyOperator}.

Similar to classical Perron's method for the existence of solutions, our strategy is to take either the pointwise supremum of all subsolutions or the pointwise infimum of all supersolutions. To accomplish the former, we assume the following conditions on $T$ and $\mathcal{U}$.

\begin{enumerate}
 \item[(A1)] (Monotonicity) If $u \leq v$ in $X$ then $Tu \leq Tv$ in $X$.
 \item[(A2)] (Non-emptyness for subsolutions) There exists at least one subsolution of \eqref{eq:TDPP}.
 \item[(A3)] (Boundedness of subsolutions) For a uniform constant $C = C(T) > 0$ any subsolution $u$ of \eqref{eq:TDPP} satisfies that $\sup_X u \leq C$.

\item[(A4)] (Closedness under supremum) If $\tilde{\mathcal{U}} \subset \mathcal{U}$, and 
\[
\sup_X \sup_{u \in \tilde{\mathcal{U}}} u < \infty,
\] 
then 
\[
\tilde{u}_1(x) := \sup_{u \in \tilde{\mathcal{U}}} u(x) \in \mathcal{U}.
\]
\end{enumerate}

If one chooses to consider the infimum of all supersolutions, then we should replace (A2)--(A4) above with (B2)--(B4) below.
\begin{enumerate}
\item[(B2)] (Non-emptiness for supersolutions) There exists at least one supersolution of \eqref{eq:TDPP}.
\item[(B3)] (Boundedness of supersolutions) For a uniform constant $C = C(T) > 0$ any supersolution $u$ of \eqref{eq:TDPP} satisfies that $\inf_X u \leq -C$.
 \item[(B4)] (Closedness under infimum) If $\tilde{\mathcal{U}} \subset \mathcal{U}$, and 
\[
\inf_X \inf_{u \in \tilde{\mathcal{U}}} u > -\infty,
\] 
then 
\[
\tilde{u}_2(x) := \inf_{u \in \tilde{\mathcal{U}}} u(x) \in \mathcal{U}.
\]
\end{enumerate}

\begin{theorem}[Existence, I]\label{main:thm2}
Let $X$ be a set of points and $\mathcal{U} \subset \{u: X \to \R\}$, and let $T: \mathcal{U} \to \mathcal{U}$. Assume $T$ and $\mathcal{U}$ satisfy (A1), (A2), (A3) and (A4).  Let $\underline{\mathcal{U}}$ be the set of subsolutions and 
\[
 \overline{u}(x) := \sup_{u \in \underline{\mathcal{U}}} u(x). 
\]
Then $\overline{u} \in \mathcal{U}$ is a solution to \eqref{eq:TDPP}.
\end{theorem}

\begin{theorem}[Existence, II]\label{main:thm3}
Let $X$ be a set of points and $\mathcal{U} \subset \{u: X \to \R\}$, and let $T: \mathcal{U} \to \mathcal{U}$. Assume $T$ and $\mathcal{U}$ satisfy (A1), (B2), (B3) and (B4). Let $\overline{\mathcal{U}}$ be the set of supersolutions and 
\[
 \underline{u}(x) := \inf_{u \in \overline{\mathcal{U}}} u(x). 
\]
Then $\underline{u} \in \mathcal{U}$ is a solution to \eqref{eq:TDPP}.
\end{theorem}

We only prove Theorem \ref{main:thm2}, since the argument for Theorem \ref{main:thm3} is analogous.
\begin{proof}[Proof of Theorem \ref{main:thm2}]
For any subsolution $u \in \underline{\mathcal{U}}$, by the monotonicity (A1), we get $Tu \in \underline{\mathcal{U}}$.

By (A3), all elements $u$ of $\underline{\mathcal{U}}$ are bounded by a uniform constant, thus $\overline{u} < \infty$ is well defined. By the monotonicity (A1), we have that
\[
 u \leq Tu \leq T\overline{u} \quad \mbox{in $X$, for any $u \in \underline{\mathcal{U}}$}.
\]
Taking pointwise the supremum, this implies
\[
 \overline{u} \leq T\overline{u} \quad \mbox{in $X$},
\]
and thus $\overline{u} \in \underline{\mathcal{U}}$. As observed above, this implies in particular that $T \overline{u} \in \underline{\mathcal{U}}$, and since $\overline{u}$ is the pointwise supremum, we have
\[
 T\overline{u} \leq \overline{u} \quad\text{in $X$}.
\]
Thus, $\overline{u}$ is a solution.
%
%
%
%
\end{proof}

\begin{remark}
It is straight-forward to extend the proof of Theorem~\ref{main:thm2} to a more general function space $\mathcal{U} \subset \{u: X \to Y\}$, where $Y$ is a conditionally complete lattice $Y$.
\end{remark}

%
Theorem \ref{main:thm2} can be viewed as a general version of Perron's method. The regularity assumptions on $u$ are minimal. It is also clear that if all subsolutions are bounded from above and all supersolutions are bounded from below, then $\overline{u}$ is the maximal solution of \eqref{eq:TDPP} and $\underline{u}$ is the minimal solution of \eqref{eq:TDPP}.

It thus suffices to show that $\overline{u}\leq \underline{u}$ to conclude the uniqueness of the solutions but in this generalality this remains an open problem. It is unlikely that the usual comparison principle holds in such a general setting, but it would be interesting to find conditions on $T$ when such a principle holds. 

The comparison principle is however equivalent to uniqueness:
\begin{proposition}\label{la:uniquenesVScomparison}
Assume that $T: \mathcal{U} \to \mathcal{U}$ satisfies (A1)--(A4) and (B2)--(B4). Then the following are equivalent
\begin{itemize}
 \item[(i)] For any $u_1, u_2 \in \tilde{U}$ with $Tu_i = u_i$ in $X$, $i = 1,2$, then $u_1 = u_2$ in $X$.
 \item[(ii)] For any $u_1, u_2 : X \to \R$ such that $u_1 \leq Tu_1$ and $u_2 \geq Tu_2$ we have $u_1 \leq u_2$ in $X$.
\end{itemize}
\end{proposition}
\begin{proof}
(i)$\Rightarrow$ (ii):
Let $\overline{u}, \underline{u} : X \to \R$ be the solutions obtained respectively from Theorem~\ref{main:thm2} and Theorem \ref{main:thm3}. Note that any subsolutions $u$ satisfies $u \leq \overline{u}$ and any supersolution $v$ satisfies $v \geq \underline{u}$. 
The assumption (i) implies that
\[
 u_1 \leq \overline{u} = \underline{u} \leq u_2.
\]

(ii)$\Rightarrow$ (i): Since both solutions $u_1, u_2$ are both, sub- and supersolutions, we have immediately $u_1 \leq u_2 \leq u_1$, and thus $u_1 = u_2$.
\end{proof}

Another important issue is to ensure the uniform boundedness (A3) of subsolutions (or supersolution, respectively). In the Section~\ref{s:boundedness} we discuss this in more detail.

\section{The Boundedness}\label{s:boundedness}
We discuss for subsolutions $u$ of \eqref{eq:TDPP} the boundedness from above. In \cite{LS1} we obtained boundedness for a quite general type of DPP using an iteration method. Here, we give a different type of proof for our situation: The idea is to construct a strict supersolution of \eqref{eq:TDPP} and to use a weaker type of comparison principle.

Let us first give a definition for strict sub- and supersolutions.

\begin{definition}
For any given operator $T:\mathcal{U}\to \mathcal{U}$, a function $u\in \mathcal{U}$ with $\sup_X u < \infty$ is called a \emph{strict subsolution} of \eqref{eq:TDPP} if there exists a constant $\sigma>0$ such that
\[
u\leq Tu-\sigma \text{ in $X$.}
\]
Similarly, a function $u\in \mathcal{U}$ with $\inf_X u > -\infty$ is called a \emph{strict supersolution} of \eqref{eq:TDPP} if there exists a $\sigma>0$ such that
\[
u\geq Tu+\sigma \text{ in $X$.}
\] 
\end{definition}
It is clear that a strict subsolution (resp., strict supersolution) is a subsolution (resp., supersolution).

Again, the above definition in practice contains the boundary condition $u=g$ on $X\setminus\Omega$.
\begin{theorem}[Strict comparison theorem]\label{comparison:thm}
Let $\mathcal{U}$ be a linear space. Let $T:\mathcal{U}\to \mathcal{U}$ be any given operator that satisfies monotonicity (A1) of Theorem \ref{main:thm2}.  Assume that for any $x\in X$, any $c\in \mathbb{R}$ and any $u\in \mathcal{U}$, the following two relations holds:
\begin{equation}\label{tran:free}
T(u+c)(x)\leq (Tu)(x)+c,
\end{equation}
Then any subsolution $u\in \mathcal{U}$ of \eqref{eq:TDPP} and any \emph{strict} supersolution $v\in \mathcal{U}$ of \eqref{eq:TDPP} 
satisfies
\[
 u \leq v \quad \mbox{in $X$}.
\]
The same relation holds if $u$ is a strict subsolution and $v$ is a supersolution.
\end{theorem}

\begin{proof}[Proof of Theorem \ref{comparison:thm}]
By definition, we have
\begin{equation}\label{eq:usubsolution}
u\leq Tu \quad\text{in $X$}
\end{equation}
and there exists $\sigma > 0$ such that
\begin{equation}\label{eq:vstrict}
 v \geq Tv + \sigma \quad \mbox{in $X$.}
\end{equation}
Suppose by contradiction that for some $m \in (0,\infty)$
\[
\sup_{X} (u-v)=m.
\]
Then we take $\hat{v}:=v+m$. By \eqref{tran:free},$\hat{v}$ is also a strict supersolution. Indeed, 
\[
 v(x)+m \geq (Tv)(x)+m +\sigma \geq T(v +m)(x) +\sigma.
\]
Also, by \eqref{eq:vstrict},
\begin{equation}\label{eq:tran}
T\hat{v}(x)\leq Tv(x)+m \quad \text{for any $x\in X$.}
\end{equation}
%
%
%
Moreover,
\[
\sup_X (u-\hat{v})=0.
\]
Then for any $\delta >0$ there exists $x_\delta\in \Omega$ such that
\begin{equation}\label{eq:comparison1}
u(x_\delta)-\hat{v}(x_\delta)\geq -\delta.
\end{equation}
On the other hand, $u(x)\leq \hat{v}(x)$ for all $x\in X$.
The monotonicity condition (A1) implies that
\begin{equation}\label{eq:comparison2}
Tu(x_\delta)\leq T\hat{v}(x_\delta).
\end{equation}
Combining the inequalities \eqref{eq:usubsolution}, \eqref{eq:tran}, \eqref{eq:comparison1} and \eqref{eq:comparison2}, we get
\[
v(x_\delta) + m = \hat{v}(x_\delta)\leq u(x_\delta)+\delta \leq Tu(x_\delta)+\delta\leq T\hat{v}(x_\delta)+\delta \leq T v(x_\delta) + m + \delta,
\]
and consequently
\[
 v(x_\delta) \leq Tv(x_\delta) + \delta.
\]
which contradicts \eqref{eq:vstrict} if we choose $\delta<\sigma$.
\end{proof}

In order to obtain the boundedness of all subsolutions in terms of the given operator $T$, it is therefore important to build a strict supersolution that is bounded from above. 
\begin{corollary}[Boundedness]\label{prop boundedness}
Suppose that $T:\mathcal{U}\to \mathcal{U}$ is an operator that satisfies the conditions of Theorem \ref{main:thm2}. Assume that there exists $C\in \mathbb{R}$ depending on $T$ and a strict supersolution $v\in \mathcal{U}$ of \eqref{eq:TDPP} satisfying $v\leq C$. Then $u\leq C$ for all subsolutions $u$ of \eqref{eq:TDPP}. Similarly, if there exists $C\in \mathbb{R}$ depending on $T$ and a strict subsolution $v\in \mathcal{U}$ of \eqref{eq:TDPP} satisfying $v\geq -C$, then $u\geq -C$ for all supersolutions $u$ of \eqref{eq:TDPP}.
\end{corollary}
This result follows immediately from Theorem \ref{comparison:thm}. However, there seems to be no universal method to get the existence of a bounded strict supersolution for a general $T$. One needs to discuss it case by case.

\section{Comparison Principle: A Special Case}\label{sec uniqueness}
For general discrete dynamic programming equations, the uniqueness problem is challenging. In \cite{LS1}, we showed that in the case of the infinity laplacian as in \eqref{eq:DPPinfty}, uniqueness follows if one assumes running costs. Technically, there it was shown that a certain discretized flow converges uniformly to the solution, which implies as an immediate corollary the uniqueness.

Here we show, that running costs imply uniqueness in our more general context. Uniqueness follows from the comparison result, Theorem \ref{comparison:thm}, by approximating a supersolution (or subsolution) with a strict supersolution (or strict subsolution). This method works well especially when the corresponding games bear positive or negative running payoffs. The idea in what follows is inspired by \cite{I2}.

\begin{theorem} \label{comparison:thm2}
Let $\mathcal{U}$ be a linear space. Let $T:\mathcal{U}\to \mathcal{U}$ be any given operator that satisfies monotonicity (A1) of Theorem \ref{main:thm2}.  Assume that for any $x\in X$, any $c\in \mathbb{R}$ and any $u\in \mathcal{U}$, \eqref{tran:free} holds.
In addition let $T$ satisfy
\begin{enumerate}
\item[(A5)] for any $u\in \mathcal{U}$ and any $\lambda>1$, there exists $\sigma>0$ such that 
\[
T(\lambda u)\leq \lambda Tu-\sigma.
\]
\end{enumerate}
Then any subsolution $u\in \mathcal{U}$ of \eqref{eq:TDPP} and any supersolution $v\in \mathcal{U}$ of \eqref{eq:TDPP} 
satisfies
\[
 u \leq v \quad \mbox{in $X$}.
\]
\end{theorem}
\begin{proof}
Suppose $u$ and $v$ are respectively a subsolution and a supersolution of \eqref{eq:TDPP}. It is easily seen that under the assumption (A5), $\lambda v$ is a strict supersolution for any $\lambda>1$. Indeed, there exists $\sigma>0$ depending on $\lambda$ and $v$ such that
\[
\lambda v\geq \lambda Tv\geq T(\lambda v)+\sigma.
\]
Thus, using Theorem \ref{comparison:thm}, we have $u\leq \lambda v$ for any $\lambda>1$. We conclude the proof by sending $\lambda\to 1$.
\end{proof}

In order to use Theorem~\ref{comparison:thm2}, an extra step may be needed in practical use: (A5) is not necessarily satisfied if $T$ assigns nonpositive values on the boundary. 
In this case, one only needs to shift the boundary value up to make it positive in the definition of $T$. The uniqueness of solutions then follows immediately. See the examples below for more details.

\section{Examples of Dynamic Programming Principle}\label{section:example}
We give several typical examples in $\mathbb{R}^n$. Hereafter $B_\eps(x)$ denotes the ball centered at $x$ with radius $\eps>0$ while $B_\eps$ simply means $B_\eps(0)$. We first consider (degenerate) elliptic equations in a bounded domain $\Omega\subset \mathbb{R}^n$ with Dirichlet boundary condition $g\in C(\partial\Omega)$. In our exposition below, it is enough to assume that $g$ is bounded. To connect the results in the previous sections, we let $X=\overline{\Omega}=\Omega\cup\partial\Omega$.

\begin{example}\label{ex infty laplace general}
A more general variant of the example discussed in Introduction is related to the so called \emph{biased tug-of-war games}. More precisely, one considers the problem
\begin{numcases}{}
-\Delta_\infty u+c |\nabla u| = f(x) & \text{  in $\Omega$,}\\
u=g &\text{ on $\partial\Omega$,}
\end{numcases}
where $c\geq 0$ is a fixed constant and $f$ is assumed to be a bounded function on $\overline{\Omega}$. This PDE is also investigated in \cite{ASS} with mixed boundary conditions. The associated DPP is discussed in \cite{PPS} for the case $f\equiv 0$ using games and in \cite{LS1} for the case $\min_{\overline{\Omega}}f>0$ with a tree approach. The DPP for the value function is given by
\begin{numcases}{}\label{infty lap dpp}
u^\eps(x)= \mu\sup_{y\in B_\eps(x)} u^\eps(y)+(1-\mu)\inf_{y\in B_\eps(x)} u^\eps(y) +{f(x)\over 2} \eps^2 \ &\text{in $\Omega$},\\
u^\eps(x) = g(x)  &\text{on $\partial \Omega$},
\end{numcases}
where $\mu={1/2}-{c\eps/4}$ with $\eps>0$ small such that $\mu\in (0, 1)$. 


We take $\mathcal{U}$ to be the set of all functions $\overline{\Omega}\to \mathbb{R}$, which is clearly closed under the operators of supremum and infimum. Let 
\[
Tu(x):=\left\{
\begin{aligned}
& \mu\sup_{y\in B_\eps(x)} u(y)+(1-\mu)\inf_{y\in B_\eps(x)} u(y) +{f(x)\over 2} \eps^2 &&\text{for $x\in \Omega$}\\
& g(x)   &&\text{for $x\in \partial\Omega$,}
\end{aligned}
\right.
\]
with $\mu={1/2}-{c\eps/4}$ with $\eps>0$ small such that $\mu\in (0, 1)$,
It is not difficult to see that $T$ satisfies (A1) in Theorem \ref{main:thm2} and \eqref{tran:free} as well. The existence of a subsolution is also easily obtained by taking $u\equiv \inf_{\partial\Omega} g$.

In order to show the boundedness property (A3), we construct the following strict supersolution $v$. Suppose the diameter of domain $\Omega$ is $D$. Let $\sigma>0$.
We divide the domain $\Omega$ into subregions according to the distance to the boundary $\partial\Omega$.
Set
\begin{equation}\label{partition}
\begin{aligned}
&\Omega_0:=\mathbb{R}^n\setminus \Omega, \\
&\Omega_k:=\{x\in \Omega: (k-1)\eps<\dist(x, \partial\Omega)\leq k\eps\} \text{ for all $k=1, 2, \dots, N$},
\end{aligned}
\end{equation}
where $N<\infty$ is the total number for the partition, as $\Omega$ is bounded. It is clear that $\Omega=\bigcup_{k=1}^N\Omega_k$. 
We then define a bounded function $v\in \mathcal{U}$ as follows:
\[
v(x):=-a(e^{-k\eps}-1)+\sup_{\partial\Omega}g+1 \quad \text{  when $x\in \Omega_k$ for any $k=0, 1, 2, \dots, N$}
\]
where $a>0$ is determined later. Note that $v(x)\geq g(x)+1=T(x)+1$ for any $x\in \partial\Omega$ and
\[
\begin{aligned}
v(x)-Tv(x)=&\ v(x)-\mu\sup_{B_\eps(x)} v-(1-\mu)\inf_{B_\eps(x)} v-{f(x)\over 2}\eps^2\\
=&\ ae^{-(k-1)\eps}\left(-e^{-\eps}+\mu e^{-2\eps}+(1-\mu)\right)-{f(x)\over 2}\eps^2\\
\geq &\ ae^{-D}\left(c+1\over 2\right)\eps^2-\eps^2\sup_{\overline{\Omega}} f+o(\eps^2)
\end{aligned}
\]
for all $x\in \Omega$ when $\eps>0$ is small. Since $c\geq 0$, it follows that $v$ is a strict supersolution if $a$ is large enough and therefore all of the subsolutions are bounded from above by Proposition \ref{prop boundedness}. We thus conclude that there exists a solution to \eqref{infty lap dpp}.

The solutions are unique if $\sigma_1:=\inf_{\overline{\Omega}} f>0$, as follows from the arguments in \cite{LS1}. For an alternative argument, we first notice that a constant $-C$ in $X$ is a strict subsolution for a sufficiently large $C>0$ and therefore $u\geq -C$ in $X$ for any solution $u$. It suffices to show that $\hat{u}=u+C$ uniquely solves
$\hat{u}=\hat{T}\hat{u}$ with 
\[
\hat{T}u(x):=\left\{
\begin{aligned}
& \mu\sup_{y\in B_\eps(x)} u(y)+(1-\mu)\inf_{y\in B_\eps(x)} u(y) +{f(x)\over 2} \eps^2 &&\text{for $x\in \Omega$}\\
& g(x) +C  &&\text{for $x\in \partial\Omega$,}
\end{aligned}
\right.
\]
where $C$ is a positive constant satisfying $G:=\inf_{x\in \partial\Omega}g(x)+C>0$. Note that $\hat{T}$ satisfies (A5) in Theorem \ref{comparison:thm2}:
\[
\begin{aligned}
&\hat{T}(\lambda u)(x)\leq \lambda \hat{T}u(x) +{1-\lambda\over 2}{f(x)}\leq \lambda \hat{T}u(x)-(1-\lambda)\sigma_1 &\text{for $x\in\Omega$;}\\
&\hat{T}(\lambda u)(x)=\hat{T}u(x)\leq \lambda \hat{T}u(x)-(1-\lambda)G &\text{for $x\in \partial \Omega$}.
\end{aligned}
\]
We may use Theorem \ref{comparison:thm2} to reach the conclusion.

\end{example}

\begin{example}[Stationary mean curvature operator]\label{ex mean curvature} A typical elliptic PDE involving level set mean curvature operator is as follows:
\begin{numcases}{\rm (MCF)\quad }
-|\nabla u|\dive\left(\frac{\nabla u}{|\nabla u|}\right)-1=0 &\text{ in $\Omega$,}\\
u=0 &\text{ on $\partial\Omega$,} \label{bdry smcf}
\end{numcases}
where $\Omega$ is a bounded domain in $\mathbb{R}^n$. 
A deterministic game-theoretic interpretation is also available, given in \cite{KS1}. For simplicity, we only study the case when $n=2$. The DPP is then written as 
\begin{equation}\label{mean curvature dpp}
\left\{
\begin{aligned}
& u^\eps(x)=\inf_{w\in \partial B_\eps}\sup_{b=\pm 1} u^\eps(x+\sqrt{2} bw)+\eps^2 \quad &&\text{for  $x\in \Omega$},\\
& u^\eps(x) = 0 \quad &&\text{for  $x\in \partial \Omega$}.
\end{aligned}
\right.
\end{equation}
See more details about the related games in \cite{KS1, L1, L2}. We here apply Theorem \ref{main:thm2} to seek a solution to \eqref{mean curvature dpp}. Let $\mathcal{U}$ be the set of all functions: $\overline{\Omega}\to \mathbb{R}$. The DPP in this case is again written as $u=Tu$ for any $u\in \mathcal{U}$ with  
\[
Tu(x):=\begin{cases}\displaystyle
 \inf_{w\in \partial B_\eps}\sup_{b=\pm 1} u(x+\sqrt{2} bw)+\eps^2 &\text{for $x\in\Omega$,}\\
 0 &\text{for $x\in\partial \Omega$.}
\end{cases}
\] 
It is clear that $u\equiv 0$ in $\overline{\Omega}$ is a subsolution of $u=Tu$. 
It is then easily seen that $T$ and $\mathcal{U}$ satisfy the conditions (A1), (A2) and (A4) in Theorem \ref{main:thm2}. It remains to verify (A2). We may take $R>0$ such that $\overline{\Omega}\subset B_R$.  We then take a partition for the disk $B_R$: 
\[
\mathcal{O}_1=B_{\sqrt{2}\eps}, \mathcal{O}_{k}=B_{\sqrt{2k}\eps}\setminus \mathcal{O}_{k-1}, \text{ for $k=2, 3, ...$}
\]
Without loss, we assume $R^2=2K$ for some $K\in \mathbb{N}$, which implies that $\bigcup_{k=1}^K \mathcal{O}_k=B_R$. We denote $\mathcal{O}_{K+1}=\mathbb{R}^2\setminus B_R$ Now we define $v\in \mathcal{U}$ to be
\[
v(x):=2(K+1-k)\eps^2 \text{ if $x\in \mathcal{O}_k$ for $1\leq k\leq K+1$} 
\]
and claim that $v$ is a bounded strict supersolution in $\overline{\Omega}$. The boundedness is clear. To show that $v$ is a strict supersolution, we first notice that $v(x)\geq 2\eps^2$ when $x\in \partial\Omega$. Moreover, for any $x\in {\Omega}$, there exists $1\leq k\leq K$ such that $x\in \mathcal{O}_k$ and therefore $v(x)=2(K+1-k)\eps^2$. It is clear that $x\pm \sqrt{2}w\in \mathcal{O}_{k+1}$ for $w\in \partial B_\eps(x)$ and orthogonal to $x$ ($w$ can be arbitrary in $\partial B_\eps(x)$ if $x=0$).  This yields 
\[
Tv(x)=2(K-k)\eps^2+\eps^2=v(x)-\eps^2
\]
and therefore $v(x)\geq Tv(x)+\eps^2$.

The uniqueness of solutions holds as well in this case. The proof, omitted here, is based on Theorem \ref{comparison:thm2} and similar to that in Example \ref{ex infty laplace general}. We remark that despite our solution is the same as the value function $u^\eps$ in games, it still remains as an open question whether the solution converges to a unique solution of (MCF) as $\eps\to 0$, when $\Omega$ is a general non-convex domain \cite{KS1}. 
\end{example}

\begin{example}[Discrete games for the Eikonal equations] Our general DPP even applies to the first order Hamilton-Jacobi equations, but it is a discretized version of those well studied, for instance, in \cite{BC, Ko}. We take an easy example of the Eikonal equation.
\begin{numcases}{}
|\nabla u|=f(x) &\text{ in $\Omega$},\\
u=0 &\text{ on $\partial \Omega$,}
\end{numcases}
where $f\geq 0$ is a given bounded function in $\Omega$. The DPP for the associated optimal control problem is 
\[
\left\{
\begin{aligned}
& u^\eps(x)=\inf_{y\in B_\eps(x)} u^\eps(y)+\eps f(x) &&\text{for $x\in \Omega$,}\\
& u^\eps=0 &&\text{on $\partial\Omega$}.
\end{aligned}
\right.
\]
We let $\mathcal{U}$ be the set of all functions $\overline{\Omega}\to \mathbb{R}$ again and 
\begin{equation}\label{eikonal}
Tu(x):=\begin{cases}\displaystyle
 \inf_{y\in B_\eps(x)} u(y)+\eps f(x) &\text{for $x\in\Omega$,}\\
 0 &\text{for $x\in\partial \Omega$.}
\end{cases}
\end{equation}
We only show that (A3) in Theorem \ref{main:thm2} is satisfied since (A1), (A2) and (A4) are straightforward. We adopt the same partition of $\Omega$ as in \eqref{partition} and take 
\[
v(x):= 2k(\sup_{\overline{\Omega}} f+1) \quad \text{  when $x\in \Omega_k$ for any $k=1, 2, \dots, N$}
\]
We then have $v(x)\geq 1=Tv(x)+1$ if $x\in \partial\Omega$ and 
\[
v(x)-Tv(x)=2\sup_{\overline{\Omega}}f+2-\sup_{\overline{\Omega}}f\geq 2
\] 
for every $x\in \Omega$, which implies that $v$ is a strict supersolution. By Theorem \ref{comparison:thm}, we get $u\leq v$ in $\overline{\Omega}$ for any subsolution $u$ and therefore (A3) is verified. It follows from Theorem \ref{main:thm2} that there exists a solution $u=Tu$ in $\overline{\Omega}$ for the Eikonal operator \eqref{eikonal}. The solutions are unique since $T$ satisfies (A5) in Theorem \ref{comparison:thm2} after a translation of values on the boundary, i.e., we set
\[
\hat{T}u(x):=\left\{
\begin{aligned}
& \inf_{y\in B_\eps(x)} u(y)+\eps f(x)  &&\text{for $x\in \Omega$}\\
& C  &&\text{for $x\in \partial\Omega$,}
\end{aligned}
\right.
\]
and discuss uniqueness of the solution $\hat{u}=u+C$ to $\hat{u}=\hat{T}\hat{u}$ by following the argument in Example \ref{ex infty laplace general}.

\end{example}

It is possible to study the games for time-dependent problems as well. Consult \cite{MPR2} and \cite{KS1} respectively for the games related to parabolic $p$-Laplace equations and mean curvature flow equations. Our last example is from \cite{KS2} about the deterministic games for a general parabolic equation.
\begin{example}[General fully nonlinear parabolic equations] 
Consider the parabolic equation
\begin{numcases}{}
u_t+F(x, \nabla u, \nabla^2u)=0 &\text{ in $\mathbb{R}^n\times (0, \infty)$},\\
u=g &\text{ on $\mathbb{R}^n\times \{0\}$,}
\end{numcases}
where $F$ is assumed to fulfill the following:
\begin{itemize}
\item[(F1)] $F$ satisfies the ellipticity condition.
\item[(F2)] There exists $C_1>0$ such that 
\begin{equation*}
|F(x, p, \Gamma)-F(x, p', \Gamma')|\leq C_1(|p-p'|+\|\Gamma-\Gamma'\|)
\end{equation*}
for all $x, p, p'\in \mathbb{R}^n$ and symmetric matrices $\Gamma, \Gamma'$.
\item[(F3)] There exists $C_2>0$ and $q, r\geq 1$ such that 
\[
|F(x, p, \Gamma)|\leq C_2(1+|p|^q+\|\Gamma\|^r)
\]
for all $x, p$ and $\Gamma$.
\end{itemize}  
The dynamic programming for the value function $u^\eps$ in the associated games in this case is
\[
u^\eps(x, t)=g(x) \quad \text{for $0\leq t<\eps^2$}
\]
and
\[
u^\eps(x, t)=\sup_{p,\ \Gamma}\inf_w \left[u^\eps(x+\eps w, t-\eps^2)-\eps p\cdot w-{1\over 2}\eps^2\Gamma w\cdot w-\eps^2 F(x, p, \Gamma)\right]
\]
for any $t\geq \eps^2$, $x\in \mathbb{R}^n\times (0, \infty)$, where $w, p\in\mathbb{R}^n$ with 
\begin{equation}\label{controlset1}
|w|\leq \eps^{-\alpha},\quad |p|\leq \eps^{-\beta}
\end{equation}
 and $\Gamma$ is an $n\times n$ symmetric matrix satisfying 
\begin{equation}\label{controlset2}
\|\Gamma\|\leq \eps^{-\gamma}.
\end{equation}
Here the parameters $\alpha, \beta$ and $\gamma$ are all positive constants and satisfy proper relations below:
\begin{equation}\label{coefficients}
\begin{aligned}
&\alpha+\beta<1, \quad 2\alpha+\gamma<2, \quad \max(\beta q, \gamma r)<2;\\
&\gamma<1-\alpha, \quad \beta(q-1)<\alpha+1, \quad \gamma(r-1)<2\alpha, \quad \gamma r<1+\alpha.
\end{aligned}
\end{equation}

We let $\mathcal{U}$ be the set of all functions $\mathbb{R}^n\times (0, \infty)\to \mathbb{R}$ as before and define
\begin{equation}\label{general parabolic}
\begin{aligned}
&Tu(x, t):=\\
&\begin{cases}\displaystyle
 \sup_{p,\ \Gamma}\inf_w \left[u(x+\eps w, t-\eps^2)-\eps p\cdot w-{1\over 2}\eps^2\Gamma w\cdot w-\eps^2 F(x, p, \Gamma)\right] &\text{if $t\geq \eps^2$,}\\
 g(x) &\text{if $t<\eps^2$.}
\end{cases}
\end{aligned}
\end{equation}
Let us use Theorem \ref{main:thm2} to get a solution of $u=Tu$ in $X=\mathbb{R}^n\times [0, \tau)$ for a fixed $\tau> 0$ and a small $\eps>0$. Denote $\Omega=\mathbb{R}^n\times [\eps^2, \tau)$. The condition (A1) and (A4) are clearly fulfilled. It is also not difficult to see (A2) holds, since $u_1(x, t)=-C_F t+ \inf_{X\setminus \Omega} g$ is a subsolution: for every $x\in \mathbb{R}^n$
\[
\begin{aligned}
Tu_1(x, t)&\geq \inf_w \left(u_1(x+\eps w, t-\eps^2)-\eps^2 F(x, 0, O)\right)\\
&\geq -C_F (t-\eps^2)+ \inf_{X\setminus\Omega} g-C_F\eps^2=u_1(x, t)
\end{aligned}
\]
when $\eps^2\leq t <\tau$ and $u_1(x, t)\leq g(x)$ when $0\leq t<\eps^2$.

Concerning (A3), we construct a strict supersolution $v$ in $\mathbb{R}^n\times [0, \tau)$. We first define an operator $S$ for any $x\in \mathbb{R}^n$ and any bounded function $\phi$ on $\mathbb{R}^n$:
\[
S[x, \phi]:=\sup_{p,\ \Gamma}\inf_w \left[\phi(x+\eps w)-\eps p\cdot w-{1\over 2}\eps^2\Gamma w\cdot w-\eps^2 F(x, p, \Gamma)\right]
\]
with $w, p, \Gamma$ satisfying \eqref{controlset1}--\eqref{controlset2} and \eqref{coefficients}.
Then we aim to find $v(x, t)$ such that 
\[
v(x, t)\geq S[x, v(\cdot, t-\eps^2)]+\sigma \text{ for $t\geq \eps^2$.}
\] 
with some $\sigma>0$. We adapt \cite[Lemma 4.1]{KS2} to our setting, presented below for the reader's convenience.
\begin{lemma}[Lemma 4.1 in \cite{KS2}]\label{operator est}
Assume that $F$ satisfy (F1)--(F3). Let $S$ be defined as above with \eqref{controlset1}--\eqref{controlset2} and \eqref{coefficients}.
Then for any $x\in \mathbb{R}^n$ and smooth bounded function $\phi$, 
\begin{equation}
S[x, \phi]-\phi=\eps^2 F(x, \nabla \phi, \nabla^2\phi)+o(\eps^2).
\end{equation}
The constant implied in the error term is uniform with respect to $x$.
\end{lemma}
Now fix $\sigma>0$. We show that $v(x, t):=(\sigma +C_2)t+\sup_{X\setminus \Omega} g +\sigma\eps^2$ is a strict supersolution when $\eps>0$ is small. By Lemma \ref{operator est}, for sufficiently small $\eps>0$, we then get
\[
\begin{aligned}
S[x, v(\cdot, t-\eps^2)]&=v(x, t-\eps^2)+\eps^2 F(x, 0, O)+o(\eps^2)\\
&\leq v(x, t-\eps^2)+C_2\eps^2+o(\eps^2) \leq v(x, t)-{1\over 2}\sigma\eps^2 
\end{aligned}
\]
when $\eps^2\leq t<\tau$ and clearly $v(x, t)\geq g(x)+\sigma \eps^2$ when $0\leq t<\eps^2$.

In conclusion, by Theorem \ref{main:thm2}, there exists a function $u$ in $\mathbb{R}^n\times [0, \tau)$ satisfying $u=Tu$, where $T$ is as in \eqref{general parabolic}. This solution coincides with the game value $u^\eps$ in \cite{KS2}, since the uniqueness of solutions holds, as shown in the comparison principle below.
\begin{proposition}[Comparison principle for parabolic equations]
If $u_1$ and $u_2$ are respectively a subsolution and a supersolution of $u=Tu$ in $\mathbb{R}^n\times [0, \tau)$, where $T$ is defined in \eqref{general parabolic}, then $u_1\leq u_2$ in $\mathbb{R}^n\times [0, \tau)$. 
\end{proposition}

\begin{proof}
We take $h>0$ arbitrarily and define
\[
u^h_2(x, t):= u_2(x, t)+h t+h \quad \text{ for $(x, t)\in \mathbb{R}^n\times [0, \tau)$}.
\]
It is clear that $u^h_2\to u$ uniformly as $h\to 0$. By direct calculation, we get $u^h_2(x, t)\geq Tu^h_2(x, t)+h \eps^2$ for $\eps^2\leq t<\tau$ and $u^h_2(x, t)\geq g(x)+\delta$ for $0\leq t<\eps^2$. We are led to $u_1\leq u^h_2$ in virtue of Theorem \ref{comparison:thm} and conclude by letting $h\to 0$.
\end{proof}

\end{example}

The last two examples of DPP below are related to $p$-Laplacian, for which the associated dynamic programming principle involves integrals. It is natural to include measurability into consideration when choosing the function space $\mathcal{U}$, but it is not known whether one may still obtain the closedness under supremum or infimum, as in (A4) of Theorem \ref{main:thm2}. Extra work on a modification of the notion of extrema, compatible with the measurability, seems to be necessary. We leave it as a future topic.

\begin{example}[$p$-Laplacian, I] We consider the normalized $p$-Laplace equation:
\begin{numcases}{}
-|\nabla u|^{2-p}\dive(|\nabla u|^{p-2}\nabla u)=0 &\text{ in $\Omega$},\\
u=g &\text{ on $\partial\Omega$}.
\end{numcases}
 There are two dynamic programming principles known to generate the $p$-Laplace operator. One is based on the so-called \emph{asymptotic mean value property} \cite{MPR1}:
\begin{equation}\label{asym mean}
\left\{
\begin{aligned}
& u^\eps(x)=\alpha \left({1\over 2}\sup_{y\in B_\eps(x)} u^\eps(y)+{1\over 2}\inf_{y\in B_\eps(x)} u^\eps(y)\right)+\beta\fint_{B_\eps(x)} u^\eps(y)\, dy &&\text{in $\Omega$,}\\
& u^\eps(x) = g(x)  &&\text{on $\partial \Omega$,}
\end{aligned}
\right.
\end{equation}
where $\alpha=\frac{p-2}{p+n}$ and $\beta=\frac{2+n}{p+n}$. Here one would tend to take $\mathcal{U}$ to be the space of all measurable functions, and $T$ be given by
\[
Tu(x)=\begin{cases}\displaystyle
\alpha \left({1\over 2}\sup_{y\in B_\eps(x)} u(y)+{1\over 2}\inf_{y\in B_\eps(x)} u(y)\right)+\beta\fint_{B_\eps(x)} u(y)\, dy &\text{for $x\in\Omega$},\\
 g(x) &\text{for $x\in\partial \Omega$.}
\end{cases}
\]
Nevertheless, it is not obvious how to show the closedness of $T$. For $\alpha > 0$, the boundedness follows from the arguments in \cite{LS1}.
\end{example}

\begin{example}[$p$-Laplacian, II] The work by Peres and Sheffield \cite{PS} gave another game-theoretic approach to describe $p$-harmonic functions. The authors did not provide the DPP but a suggested version is $u=Tu$ with $T$ defined below:
\[
\begin{aligned}
&Tu(x)=\\
&\begin{cases}\displaystyle
\frac{1}{2} \sup_{v\in B_\eps} \int u(x+v+z)\ d\mu_v(z) +\frac{1}{2} \inf_{v\in B_\eps} \int u(x+v+z) d\mu_v(z) &\text{for $x\in\Omega$},\\
 g(x) &\text{for $x\in\partial \Omega$,}
\end{cases}
\end{aligned}
\]
where $\mu_v$ is the uniform distribution on the sphere $S_v$ orthogonal to $v$ with radius 
\[
r=\sqrt{n-1\over p-1}.
\]
Here the verification of boundedness and closedness is a challenge.
\end{example}

\bibliographystyle{abbrv}%
\bibliography{bib_liu}%

\end{document}